\documentclass[11pt,reqno]{amsart}

\usepackage{hyperref}
\usepackage{graphicx}
\usepackage{color}
\usepackage[all]{xy}
\usepackage{amsfonts,amssymb}
\usepackage{bbm} 

\usepackage{mathrsfs}

\usepackage{pdfsync}
\usepackage{marginnote}
\usepackage{a4wide}
\usepackage{verbatim}

\newtheorem{neu}{}[section]
\newtheorem{Cor}[neu]{Corollary}
\newtheorem*{Cor*}{Corollary}
\newtheorem{Thm}[neu]{Theorem}
\newtheorem*{Thm*}{Theorem}
\newtheorem*{Observation*}{Observation}
\newtheorem{Prop}[neu]{Proposition}
\newtheorem*{Prop*}{Proposition}
\theoremstyle{definition}
\newtheorem{Lemma}[neu]{Lemma}
\newtheorem*{Rmk*}{Remark}
\newtheorem{Rmk}[neu]{Remark}

\newtheorem*{Ex*}{Example}

\newtheorem*{Qu*}{Question}
\newtheorem{Claim}{Claim}
\newtheorem{Def}[neu]{Definition}
\newtheorem{Conv}[neu]{Convention}

\newcommand{\N}{\mathbb{N}}

\newcommand{\Z}{\mathbb{Z}}
\newcommand{\R}{\mathbb{R}}
\newcommand{\C}{\mathbb{C}}
\newcommand{\CP}{\C\mathrm{P}}

\newcommand{\pf}{\longrightarrow}

\newcommand{\CZ}{\mu_{\mathrm{CZ}}}

\newcommand{\Morse}{\mu_{\mathrm{Morse}}}

\newcommand{\om}{\omega}
\newcommand{\Om}{\Omega}

\renewcommand{\O}{\mathcal{O}}

\newcommand{\A}{\mathscr{A}}
\renewcommand{\a}{\mathfrak{a}}

\newcommand{\U}{\mathcal{U}}

\newcommand{\M}{\mathcal{M}}
\newcommand{\Mh}{\widehat{\mathcal{M}}}
\newcommand{\Nh}{\widehat{\mathcal{N}}}

\renewcommand{\j}{\mathfrak{j}}
\newcommand{\B}{\mathcal{B}}

\renewcommand{\L}{\mathscr{L}}
\newcommand{\Lt}{\widetilde{\mathscr{L}}}
\renewcommand{\H}{\mathrm{H}}

\newcommand{\RFH}{\mathrm{RFH}}
\newcommand{\SH}{\mathrm{SH}}
\newcommand{\RFC}{\mathrm{RFC}}

\newcommand{\QH}{\mathrm{QH}}
\newcommand{\Crit}{\mathrm{Crit}}

\newcommand{\m}{\mathfrak{m}}

\newcommand{\beq}{\begin{equation}}
\newcommand{\beqn}{\begin{equation}\nonumber}
\newcommand{\eeq}{\end{equation}}

\newcommand{\bea}{\begin{equation}\begin{aligned}}
\newcommand{\bean}{\begin{equation}\begin{aligned}\nonumber}
\newcommand{\eea}{\end{aligned}\end{equation}}

\numberwithin{equation}{section}

\newcommand{\p}{\partial}

\renewcommand{\c}{\mathfrak{C}}

\begin{document}
\title[Vanishing of RFH on negative line bundles]{Vanishing of Rabinowitz Floer homology on negative line bundles}
\author{Peter Albers}
\author{Jungsoo Kang}
\address{
	Peter Albers, Jungsoo Kang\\ 
	Mathematisches Institut\\
	Westf\"alische Wilhelms-Universit\"at M\"unster}
\email{peter.albers@wwu.de, jungsoo.kang@me.com}
\begin{abstract}
Following \cite{Frauenfelder_Habilitation,Albers_Frauenfelder_Bubbles_and_Onis} we construct Rabinowitz Floer homology for negative line bundles over symplectic manifolds and prove a vanishing result.

In \cite{Ritter_Floer_theory_for_negative_line_bundles_via_Gromov_Witten_invariants} Ritter showed  that symplectic homology of these spaces does not vanish, in general. Thus, the theorem $\mathrm{SH}=0\Leftrightarrow\RFH=0$, \cite{Ritter_Topological_quantum_field_theory_structure_on_symplectic_cohomology}, does \emph{not} extend beyond the symplectically aspherical situation. We give a conjectural explanation in terms of the Cieliebak-Frauenfelder-Oancea long exact sequence \cite{Cieliebak_Frauenfelder_Oancea_Rabinowitz_Floer_homology_and_symplectic_homology}.
\end{abstract}
\maketitle

{\let\thefootnote\relax\footnotetext{\today}}

\section{Introduction}

Negative line bundles give rise to a rather special class of contact manifolds which nevertheless contains many interesting examples. They arise at many places in modern contact and symplectic geometry such as Givental's nonlinear Maslov index \cite{Givental_The_nonlinear_Maslov_index} and more generally contact rigidity \cite{Eliashberg_Polterovich_Partially_ordered_groups_and_geometry_of_contact_transformations,Sandon_Equivariant_homology_for_generating_functions_and_orderability_of_lens_spaces,Borman_Quasi_states__quasi_morphisms_and_the_moment_map,Borman_Zapolsky_Quasimorphisms_on_contactomorphism_groups_and_contac_rigidity} etc.

Let us be more specific. We choose a closed connected symplectic manifold  $(M,\om)$ with integral symplectic form $[\om]\in\H^2(M,\Z)$. We denote by $\wp:\Sigma\to M$ the principal $S^1$-bundle and by $\wp:E\to M$ the associated complex line bundle with first Chern class $c_1^E=-[\om]$. We refer to these bundles as {\em negative line bundles}. There exists an $S^1$-invariant 1-form $\alpha$ on $\Sigma$, and hence $E\setminus M$, with the property
\beq
d\alpha=\wp^*\om
\eeq
which is a contact form on $\Sigma$. For more details we refer to \cite[Section 7.2]{Geiges_book}. If we denote by $r$ the radial coordinate on $E$ then the 2-form
\beq
\Om:=d\big(\pi r^2\alpha\big)+\wp^*\om=2\pi rdr\wedge\alpha+\big(\pi r^2+1\big)\wp^*\om
\eeq
is a symplectic form on $E$.  Throughout this article we make the assumption that $(E,\Om)$ is semi-positive, see \cite[Definition 6.4.1]{McDuff_Salamon_J_holomorphic_curves_and_symplectic_topology} and page \pageref{test} for an equivalent formulation. We denote by $\Sigma_\tau$ the circle subbundle of $E$ of radius $r=\sqrt{\tau/\pi}$.

\begin{Thm}\label{thm:RFH_well_defined}
The Rabinowitz Floer homology $\RFH(\Sigma_\tau,E)$ is well-defined for all $\tau>0$.
\end{Thm}

In many situations we are able to prove the following vanishing result.

\begin{Thm}\label{thm:main}
We assume that one of the following is satisfied.
\begin{enumerate}
\item $(M,\om)$ is symplectically aspherical: $\om\big(\pi_2(M)\big)=0$.
\item There exists a constant $c$ such that
\beq\label{eqn:assumption_c_1_TM=c_om}
c_1^{TM}=c\om:\pi_2(M)\to\Z
\eeq
and $2c\nu\leq -\dim M$ holds, where $\nu\in\Z_{>0}$ is defined by $\om\big(\pi_2(M)\big)=\nu\Z$. 
\item The equality \eqref{eqn:assumption_c_1_TM=c_om} holds with $c\geq 1$. 

\end{enumerate}
Then the Rabinowitz Floer homology of $(\Sigma_\tau,E)$ vanishes,
\beq
\RFH(\Sigma_\tau,E)=0,
\eeq  
for all $\tau>0$ in the case of (1), (2), and (3) with $c=1$. The same assertion holds in the case of (3) provided $\tau<\frac{1}{c-1}$. 
\end{Thm}
In the setting of the above theorem, $(E,\Om)$ is semi-positive. Then condition $\tau<\frac{1}{c-1}$ in the case of (3) is indispensable, see Remark \ref{rmk}.(6) below.

Contact manifolds such as spheres, projective spaces etc.~come from negative line bundles. In fact, the Boothby-Wang theorem \cite[Theorem 7.2.5]{Geiges_book} characterizes these contact manifolds as those whose Reeb flow is periodic with all Reeb orbits having the same minimal period. On the sphere this corresponds to the Hopf fibration. 

\begin{Rmk}\label{rmk} $ $
\begin{enumerate}
\item If $c=0$, i.e.~$c_1^{TM}\big(\pi_2(M)\big)=0$, then $c_1^{TE}=-\wp^*\om:\pi_2(E)\to\Z$. Therefore, $(E,\Om)$ is semi-positive if and only if $\nu\geq\tfrac12\dim M-1$ or $\om\big(\pi_2(M)\big)=0$, see Lemma \ref{lem:semi_positive}. If $(E,\Om)$ is semi-positive then $\RFH(\Sigma,E)=0$ still holds.


\item It is worth pointing out that $\Sigma$ is \emph{not} displaceable inside $E$ since the zero-section $M\subset E$ is not even topologically displaceable. To our knowledge this is the first vanishing result for $\RFH$ result which is \emph{not} due to a displaceability phenomenon, see also Ritter \cite[Remark on p.~1044]{Ritter_Floer_theory_for_negative_line_bundles_via_Gromov_Witten_invariants}. 

\item There are very few \emph{direct} computations of $\RFH$. To our knowledge, the vanishing result in the  displaceable case ~\cite{Cieliebak_Frauenfelder_Restrictions_to_displaceable_exact_contact_embeddings}, the computation for cotangent bundles  ~\cite{Abbo_Schwarz_Estimates_and_computations_in_RFH}, and the computation for  Brieskorn spheres ~\cite{Fauck_RFH_on_Brieskorn_spheres} are the only ones. The long exact sequence, \cite{Cieliebak_Frauenfelder_Oancea_Rabinowitz_Floer_homology_and_symplectic_homology} leads often to computational results if the symplectic homology and the connecting maps are known. The latter rarely happens, though. 

\item Rabinowitz Floer homology, first constructed by Cieliebak and Frauenfelder in \cite{Cieliebak_Frauenfelder_Restrictions_to_displaceable_exact_contact_embeddings}, is an invariant of contact type hypersurfaces in symplectic manifolds. It turned out to be an efficient tool for studying questions in symplectic topology and dynamics, see \cite{Albers_Frauenfelder_RFH_Survey}. 

In \cite{Frauenfelder_Habilitation} Frauenfelder studied the Rabinowitz Floer homology of negative line bundles under the additional assumption of the line bundle being very negative. The implication of the latter is the generic absence of holomorphic spheres. Amongst many other things he established $C^\infty_{loc}$-compactness results, cf.~\cite[Theorem B]{Frauenfelder_Habilitation}. Even though Rabinowitz Floer homology is not fully constructed in \cite{Frauenfelder_Habilitation} all ingredients are basically contained therein, see also \cite{Albers_Frauenfelder_Bubbles_and_Onis}.  

The purpose of this article is to complete and extend the construction of Rabinowitz Floer homology to negative line bundles in the presence of holomorphic spheres under a semi-positivity assumption. In particular, we prove a transversality result made necessary due to the use of a rather restricted class of almost complex structures. 

It is worth pointing out that this is the first instance where Rabinowitz Floer homology is constructed in the presence of holomorphic spheres. Holomorphic spheres are a source for interesting symplectic topology and big technical problems at the same time. The latter is the reason we require semi-positivity. 

\item The main new contribution is Theorem \ref{thm:main}: Rabinowitz Floer homology vanishes in many cases. This should be contrasted with Ritter's result that symplectic homology does not necessarily vanishes, see \cite{Ritter_Floer_theory_for_negative_line_bundles_via_Gromov_Witten_invariants}. Thus, the theorem $\mathrm{SH}=0\Leftrightarrow\RFH=0$, \cite{Ritter_Topological_quantum_field_theory_structure_on_symplectic_cohomology}, does \emph{not} extend beyond the symplectically aspherical situation. We give a conjectural explanation of this in section \ref{sec:conj_explanation} below.

For very negative line bundles Ritter in \cite[Theorem 8]{Ritter_Floer_theory_for_negative_line_bundles_via_Gromov_Witten_invariants} proved vanishing of symplectic homology. If we assume in addition that  $c_1^{TM}=c\om:\pi_2(M)\to\Z$ the Rabinowitz Floer homology vanishes according to Theorem \ref{thm:main} as well. The conjectural picture from section \ref{sec:conj_explanation} nicely relates these results.
\item (\textbf{Updated on December 5, 2025}) In the first version of this paper, the assumption $\tau<\frac{1}{c-1}$ in case (3) was not made in the statement although this was crucially used in the proof and explicitly stated in Lemma \ref{lem:c_great_equal_0_implies_CZ_downstairs_bounded_from_above}. In fact, the assertion $\RFH(\Sigma_\tau,E)=0$ is false without this assumption.	 We refer the reader \cite{AK23}, a sequel to this paper, see in particular Remark 1.8.(a) therein. We also refer to \cite{Ven18,Ven21} for the relation between Rabinowitz Floer homology and (completed) symplectic homology, which clarifies our conjectural explanation in section \ref{sec:conj_explanation}.
\end{enumerate}
\end{Rmk}

\subsection*{Acknowledgments} 
We thank Urs Frauenfelder for illuminating discussion on the present article. PA is supported by SFB 878. JK is supported by DFG grant KA 4010/1-1.

\section{Rabinowitz Floer homology and Hamiltonian Floer homology}\label{sec:def_of_RFH}

\subsection{Preliminaries}
Let $\wp:E\to M$ be as described above. We denote by $\L(E)$ the component of contractible loops of the free loop space of $E$. Moreover, we denote by $\Lt(E)$ the covering space of $\L(E)$ with deck transformations given by 
\beq
\Gamma_E:=\frac{\pi_2(E)}{\ker\Om\cap\ker c_1^{TE}}\;.
\eeq
We write elements in $\Lt(E)$ as $[u,\bar u]$, where $u:S^1\to E$ and $\bar u:D^2\to E$ is a capping disk for $u$, i.e.~$\bar u|_{S^1}=u$. Moreover, pairs $(u,\bar u)$ and $(v,\bar v)$ are equivalent if $u=v$ and $\Om(-\bar u\#\bar v)=c_1^{TE}(-\bar u\#\bar v)=0$, where $-\bar u\#\bar v$ is the sphere formed by $\bar u$ with orientation reversed and $\bar v$. The expression $[u,\bar u]$ denotes the corresponding equivalence class. Analogously we define
\beq
\Gamma_M:=\frac{\pi_2(M)}{\ker\om\cap\ker c_1^{TM}}\;,
\eeq
$\L(M)$ and $\Lt(M)$. By definition of $E$ we have
\beq
c_1^{TE}=\wp^*\big(c_1^{TM}+c_1^E\big)\;.
\eeq

\begin{Rmk}\label{rmk:Gamma_E=Gamma_M}
We point out that under assumption \eqref{eqn:assumption_c_1_TM=c_om}, i.e.~$c_1^{TM}=c\om:\pi_2(M)\to\Z
$, we have $\ker\om\cap\ker c_1^{TM}=\ker \om$ and $\ker\Om\cap\ker c_1^{TE}=\ker \Om$. If we instead assume $\om\big(\pi_2(M)\big)=0$ then $c_1^{TE}=\wp^* c_1^{TM}:\pi_2(E)\to\Z$. In particular, since $\pi_2(E)\cong\pi_2(M)$ via $\wp_*$, we can identify in both cases $\Gamma_E\cong\Gamma_M$.
\end{Rmk}

For $\tau>0$ we denote by $\mu_\tau:E\to\R$ the function $\mu_\tau=\pi  r^2-\tau$ where as above $r$ denotes the radial coordinate on $E$. We point out that along $\Sigma_\tau:=\{\mu_\tau=0\}$ the Hamiltonian vector field $X_{\mu_\tau}$ of $\mu_\tau$ agrees with the Reeb vector field $R$ associated to the contact form $\alpha$. In particular, we use the convention $\Om(X_{\mu_\tau},\cdot)=-d\mu_\tau$. The Rabinowitz action functional $\A^\tau$ is defined as
\bea
\A^\tau:\Lt(E)\times\R&\to\R\\
\big([u,\bar u],\eta\big)&\mapsto \int_{D^2}\bar u^*\Om-\eta\int_0^1 \mu_\tau\big(u(t)\big) dt \;.
\eea
In the article \cite{Cieliebak_Frauenfelder_Restrictions_to_displaceable_exact_contact_embeddings} Cieliebak and Frauenfelder developed a Floer theory for this functional in a slightly simpler set-up and for non-degenerate contact forms. The current set-up has been developed and studied Frauenfelder in \cite{Frauenfelder_Habilitation} for very negative line bundles, see also \cite{Albers_Frauenfelder_Bubbles_and_Onis}. 

In our setting the contact form is Morse-Bott non-degenerate. A general Morse-Bott approach to Rabinowitz Floer homology is currently not available in the literature. Instead of perturbing the contact form we choose the following perturbation.  We fix $f:M\to\R$ and set 
\beq
F:=\big(\pi r^2+1\big)f\circ \wp:E\to\R\;.
\eeq
The perturbed Rabinowitz action functional is 
\bea\label{eqn:def_of_RFH_functional}
\A_f^\tau:\Lt(E)\times\R&\to\R\\
\big([u,\bar u],\eta\big)&\mapsto \int_{D^2}\bar u^*\Om-\eta\int_0^1 \mu_\tau\big(u(t)\big) dt-\int_0^1F\big(u(t)\big)dt \;.
\eea
Critical points of $\A_{f=0}^\tau$ correspond to capped Reeb orbits traversed in forward and backward direction and, in addition, to constant loops contained in $\Sigma_\tau$ together with cappings. In Lemma \ref{lem:crit_A_f} we show that for $C^2$-small Morse functions $f$ the critical points of $\A_f^\tau$ correspond to capped Reeb orbits which lie via $\wp:E\to M$ over $\Crit(f)$. The functional $\A_f^\tau$ is still Morse-Bott due to the remaining $S^1$-symmetry. This can be dealt with as in the article by Bourgeois-Oancea \cite{Bourgeois_Oancea_Symplectic_homology_autonomous_Hamiltonians_and_Mores_Bott_moduli_spaces}. 

\begin{Rmk}\label{rmk:splitting_TE=V+H}
We split the tangent bundle $TE\cong V\oplus H$ in vertical resp.~horizontal subspaces $V$ resp.~$H$ according to $\alpha$. In particular, $\wp_*:(H,d\alpha)\stackrel{\cong}{\pf}(TM,\om)$ is an isomorphism of symplectic vector bundles and $V$ is spanned by the Reeb vector field $R$ and the radial vector field $r\partial_r$.  
\end{Rmk}

\begin{Lemma}\label{lem:crit_A_f}
If $f:M\to\R$ is $C^2$-small then $\big([u,\bar u],\eta\big)$ is a critical point of $\A_f^\tau$ if and only if the following equations are satisfied.
\begin{equation}\label{eq:crit_eq}
\left\{
\begin{aligned}
&\pi r(u)^2=\tau\\
&q:=\wp(u)\in\Crit(f)\\
&\dot{u}=\big(\eta +f(q)\big)R(u)\\
\end{aligned} 
\right.
\end{equation}
In particular, necessarily $\eta+f(q)\in\Z$ and $u\subset\Sigma_\tau$ is a $(\eta+f(q))$-fold cover of the underlying simple periodic  orbit.
\end{Lemma}

\begin{proof}
The critical point equation for $\A_f^\tau$ is
\beq\label{eqn:critical_point_equation_A_f_tau}
\left\{
\begin{aligned}
\;&\dot u=\eta X_{\mu_\tau}(u)+X_F(u)\\
&\int_0^1\mu_\tau(u)dt=0\;.
\end{aligned}
\right.
\eeq
The Hamiltonian vector field of $F$ is $X_F=(f\circ\wp) X_{\mu_\tau}+X_f^h$ where $X_f^h$ is the horizontal lift of $X_f$, i.e. $X_f^h\in H$ and $\wp_*(X_f^h)=X_f$. Indeed,
\bea
\Om(X_F,\cdot)&=\Big(2\pi rdr\wedge\alpha+\big(\pi r^2+1\big)\wp^*\om\Big)\Big((f\circ\wp) X_{\mu_\tau}+X_f^h,\cdot\Big)\\
&=-2\pi r(f\circ\wp)+\big(\pi r^2+1\big)\wp^*\big(\om(X_f,\cdot)\big)\\
&=-2\pi r(f\circ\wp)-\big(\pi r^2+1\big)\wp^*df\\
&=-d\big[\big(\pi r^2+1\big)f\circ\wp\big]\\
&=-dF\;.
\eea
We point out that $\Om(X_{\mu_\tau},X_F)=dF(X_{\mu_\tau})=0$ since $\wp_*(X_{\mu_\tau})=0=dr(X_{\mu_\tau})$. Therefore the critical point equation  \eqref{eqn:critical_point_equation_A_f_tau} simplifies to 
\beq
\left\{
\begin{aligned}
\;&\dot u=\big(\eta+f(\wp(u))\big)X_{\mu_\tau}(u)+X_f^h(u)\\
&\mu_\tau\big(u(t)\big)=0\quad\forall t\in S^1\;.
\end{aligned}
\right.
\eeq
The last equation translates into $r(u(t))$ being constant and $\pi r(u)^2=\tau$. The critical point equation together with $\wp_*(X_{\mu_\tau})=0$ implies that $\wp(u)$ is a 1-periodic solution of $X_f$ in $M$. Now, if the $C^2$-norm of $f$ is sufficiently small the only 1-periodic solutions of $X_f$ are the critical points of $f$, see \cite[p.\,185]{Hofer_Zehner_Book}. Thus, $q:=\wp(u)\in\Crit(f)$ and $u$ corresponds to a $(\eta+f(q))$-periodic orbit of $R$ on $\Sigma_\tau$. This implies that $\eta+f(q)\in\Z$ due to our convention $S^1=\R/\Z$.
\end{proof}

\begin{Rmk}\label{rmk:cappings_critical_value}
To summarize critical points of $\A_f^\tau$ correspond to all Reeb orbits over $\Crit(f)$ together with cappings. More precisely, all forward (i.e.~$\eta+f(q)>0$) and backward (i.e.~$\eta+f(q)<0$) iterations and also the ``constants'' (i.e. $\eta+f(q)=0$) together with cappings.
\end{Rmk}

\begin{Conv}\label{convention_f_C2_small}\quad\\[-3ex]
\begin{itemize}\itemsep=1ex
\item From now on we assume that the Morse function $f:M\to\R$ is chosen $C^2$-small so that Lemma \ref{lem:crit_A_f} applies.
\item  Every simple periodic Reeb orbit $v\subset \Sigma_\tau$ has a capping by its fiber disk $d_v\subset E$ and correspondingly the $n$-fold cover $v^n$ has $d_v^n$ as capping disk  for $n\in\Z\setminus\{0\}$. Every non-constant critical point $\big([u,\bar u],\eta \big)$ can be expressed in the form $u=v^n$ and $\bar u=d_v^n\#A$ for some $A\in\Gamma_E$. If $u$ is a constant critical point, the capping disk $\bar u$ can be thought of as a sphere $A\in\Gamma_E$. We are going to adopt the notation $[u,\bar u]=[u,A]$. 
\end{itemize}
\end{Conv}

Using Lemma \ref{lem:crit_A_f} we compute the action value for a critical point $\big([u,\bar u],\eta\big)=\big([v^n,A],\eta\big)$ with  $q=\wp(u)$.
\bea\label{eq:computation_of_action}
\A_f^\tau\big([v^n,A],\eta\big)&=\int_{D^2}\big(d_v^n\big)^*\Om+\om(A)-\eta\int_0^1 \underbrace{\mu_\tau(v^n)}_{=0} dt-\int_0^1F(v^n)dt\\
&=\int_{D^2}\big(d_v^n\big)^*\big[d\big(\pi r^2\alpha\big)\big]-\int_0^1\big(\pi r^2+1\big)f\circ\wp(v^n)dt+\om(A)\\
&=\int_{S^1}(v^n)^*\big(\pi r^2\alpha\big)-(\tau+1)f(q)+\om(A)\\
&=\int_{S^1}\tau \alpha\big((\eta +f(q))R(v^n)\big)-(\tau+1)f(q)+\om(A)\\
&=\tau\big(\eta +f(q)\big)-(\tau+1)f(q)+\om(A)\\
&=\tau\eta+\om(A)-f(q)\\
&=\tau n+\om(A)-(\tau+1)f(q)\;,
\eea
where we use $n=\eta +f(q)$ and $\Om=\om:\Gamma_E\cong\Gamma_M\to\Z$. 

Next we explain how to define Floer homology for $\A_f^\tau$. This mainly follows the lines of \cite{Frauenfelder_Habilitation} and \cite{Albers_Frauenfelder_Bubbles_and_Onis}. We assume throughout that $(E,\Om)$ is semi-positive. According to \cite[Exercise 6.4.3]{McDuff_Salamon_J_holomorphic_curves_and_symplectic_topology} the symplectic manifold $(E,\Om)$ is semi-positive if and only if 
\begin{itemize}\itemsep=1ex
\label{test}
\item $(E,\Om)$ is symplectically aspherical, 
\item $(E,\Om)$ is monotone,
\item $c_1^{TE}:\pi_2(E)\to\Z$ vanishes,
\item the minimal Chern number $N_E$ of $E$ satisfies $N_E\geq\tfrac12\dim E-2$.
\end{itemize}
Since $\pi_2(E)\cong\pi_2(M)$ via $\wp_*$ the first condition is equivalent to $(M,\om)$ being symplectically aspherical which is condition (1) in Theorem \ref{thm:main}.

If we assume that there exists a constant $c\in\Z$ such that $c_1^{TM}=c\om:\pi_2(M)\to\Z$ then  
\beq
c_1^{TE}=\wp^*\big(c_1^{TM}+c_1^E\big)=(c-1)\wp^*\om\;.
\eeq
Thus, if $c>1$ the symplectic manifold $(E,\Om)$ is monotone and for $c=1$ we have $c_1^{TE}=0$ on $\pi_2(E)$. Furthermore, if we denote by $\nu\in\Z_{\geq0}$ the generator of $\om\big(\pi_2(M)\big)=\nu\Z$ then the minimal Chern number $N_E$ of $E$ is
\beq
N_E=|c-1|\nu\;.
\eeq
Thus, we proved the following Lemma.

\begin{Lemma}\label{lem:semi_positive}
The symplectic manifold $(E,\Om)$ is semi-positive if 
\begin{itemize}\itemsep=1ex
\item $(M,\om)$ is symplectically aspherical or
\item $c_1^{TM}=c\om:\pi_2(M)\to\Z$ with $c\geq 1$ or $N_E=|c-1|\nu\geq\tfrac12\dim E-2$.
\end{itemize}
\end{Lemma}

In the following Lemma we use the notation of Convention \ref{convention_f_C2_small}. For $\big([u,\bar u],\eta\big)=\big([u,A],\eta\big)\in\Crit(\A^\tau_f)$, we denote by $\CZ^E(u,\bar u)\equiv\CZ^E(u,A)$ the Conley-Zehnder index of $u$ with respect to the capping disk $\bar u$. We refer to \cite{Robbin_Salamon_Maslov_index_for_paths} for a thorough discussion of the Conley-Zehnder index.
\begin{Lemma}\label{lemma:index_of_fiber}
The Conley-Zehnder index of a $n$-fold cover $v^n$ 	with its fiber disk $d_v^n$ is
\beq
\CZ^E(v^n,d_v^n)=2n\;.
\eeq 
More generally, for any capping $\bar u=d_v^n\#A$ of $u=v^n$,
\beq
\CZ^E(v^n,d_v^n\#A)=2n+2c_1^{TE}(A)\;.
\eeq 
If $(M,\om)$ is symplectically aspherical then all iterates $v^n$ are non-contractible inside $\Sigma$. Otherwise, the first iterate of $v$, which is contractible in $\Sigma$, is the orbit $v^{\nu}$. 

If we assume $c_1^{TM}=c\om$ then the Conley-Zehnder index of $v^n$ for $n\in\nu\Z\setminus\{0\}$ with respect to a capping disk $\bar\nu$ contained entirely in $\Sigma$ is
\beq
\CZ^E(u,\bar \nu)=2cn\;.
\eeq
\end{Lemma}

\begin{proof}
Since the linearized map of the Reeb flow is the identity in horizontal directions, the first assertion follows from the corresponding computation for $S^1\subset\C$.  The relevant bit of the homotopy long exact sequence of the $S^1$-bundle $\Sigma\to M$ is 
\beq
\cdots\pf\pi_2(M)\stackrel{\delta}{\pf}\pi_1(S^1)\stackrel{i_*}{\pf}\pi_1(\Sigma)\pf\cdots\;.
\eeq
If we identify $\pi_1(S^1)\cong\Z$ then $i_*(k)=[v^k]$ and $\delta(s)=-\om(s)$ with respect to the homomorphism $\om:\pi_2(M)\to\Z$. Thus, if $(M,\om)$ is symplectically aspherical then all iterates $v^k$ are non-contractible in $\Sigma$. Otherwise $\om\big(\pi_2(M)\big)=\nu\Z$ with $\nu>0$ and therefore the first iterate of $v$ which is contractible in $\Sigma$ is the orbit $v^{\nu}$. 

Now we assume that $c_1^{TM}=c\om$ and recall that $\bar\nu$ is a capping which is entirely contained inside $\Sigma$. 
\bea
\CZ^E(v^n,\bar\nu)&\stackrel{\phantom{(\ast)}}{=}\CZ^E\Big(v^n,d_v^n\#\big(\underbrace{\bar\nu\#-d_v^n}_{\in\Gamma_E}\big)\Big)\\
&\stackrel{\phantom{(\ast)}}{=}\CZ^E\big(v^n,d_v^n\big)+2c_1^{TE}\big(\bar\nu\#-d_v^n\big)\\
&\stackrel{\phantom{(\ast)}}{=}2n+2(c-1)\wp^*\om\big(\bar\nu\#-d_v^n\big)\\
&\stackrel{\phantom{(\ast)}}{=}2n+2(c-1)\bigg[\int_{\bar\nu}\wp^*\om-\underbrace{\int_{d_v^n}\wp^*\om}_{=0}\bigg]\\
&\stackrel{(\ast)}{=}2n+2(c-1)\int_{\bar\nu}d\alpha\\
&\stackrel{\phantom{(\ast)}}{=}2n+2(c-1)\int_{v^n}\alpha\\
&\stackrel{\phantom{(\ast)}}{=}2n+2(c-1)n\\
&\stackrel{\phantom{(\ast)}}{=}2cn
\eea
where we used in $(\ast)$ that the disk $\bar\nu$ is contained inside $\Sigma$.
\end{proof}

\begin{Def}\label{def:index_of_critical_points}
We point out that  critical points of $\A_f^\tau$ are $S^1$-families, cf.~Lemma \ref{lem:crit_A_f}. We choose a perfect Morse function $h:\Crit(\A_f^\tau)\to\R$ such that every critical manifold $S^1\cdot\big([u,A],\eta\big)\subset\Crit(\A_f^\tau)$ gives rise to two critical points of $h$ which we denote by $\big([u,A]^\pm,\eta\big)$ according to the maximum resp.~minimum of $h$ on $S^1\cdot\big([u,A],\eta\big)$. We define the index of a critical point by
\begin{equation}\label{eq:mu-index}
\mu\big([u,A]^\pm,\eta\big):=\CZ^E(u,A)-\Morse(\wp(u),f)+\tfrac12 \dim M\pm\tfrac12 \in\tfrac12 +\Z\;,
\end{equation}
where $\Morse(\wp(u),f)$ is the Morse index of $\wp(u)\in\Crit (f)$. In case that $u$ is a constant critical point we define $\CZ^E(u,A=0):=0$. We set
\beq
\c:=\Crit(h)=\left\{\big([u,A]^\pm,\eta\big)\right\}\subset\Crit(\A_f^\tau)
\eeq
and 
\beq
\c_k:=\left\{\big([u,A]^\bullet,\eta\big)\in\c\,\bigr|\, \mu\big([u,A]^\bullet,\eta)\big)=k\right\}\;.
\eeq
Here, $\bullet$ indicates some choice of $\pm$.
\end{Def}

In order to define Rabinowitz Floer homology and to prove the vanishing result we rely on a fairly special class of almost complex structures which we describe next. In the next subsection we prove that this class is big enough to prove the necessary transversality results. We recall that we split the tangent bundle $TE\cong V\oplus H$ in vertical resp.~horizontal subspaces $V$ resp.~$H$, see Remark \ref{rmk:splitting_TE=V+H}. Let us abbreviate by
\beqn
\j:=\big\{j\in \Gamma(S^1\times M, \mathrm{Aut}(TM))\mid j_t:=j(t,\cdot) \text{ is an $\om$-compatible almost complex structure}\big\}
\eeq
the space of $S^1$-families of compatible almost complex structures on $(M,\om)$. Next we fix disjoint open balls around each point in $\Crit(f)$. The union of these balls is denoted by $\U$. For a fixed $j\in\j$ we denote by $\B(j)$ the set of $B\in \Gamma_0(S^1\times E, \mathrm{L}(H,V))$ where $B_t:=B(t,\cdot)$ satisfies
\beq\label{eqn:conditions_of_B(J)}
iB_t+B_tj_t=0\;\forall t\in S^1\text{ and }B_t(e)=0\;\forall e\in \wp^{-1}(\U)\;.
\eeq
Here the subscript $0$ indicates compact support and $\mathrm{L}(H,V)$ is the space of linear maps. To describe the Floer equation we will choose a $S^1$-family $J_t$ of  almost complex structures on $E$ of the form
\beq
J_t=\begin{pmatrix}
 i & B_t\\
 0 & j_t
\end{pmatrix}\;.
\eeq
The matrix representation refers to the splitting $TE\cong V\oplus H$. Moreover, $j\in\j$ and $B\in\B(j)$ and $i$ is the standard complex structure on $V_e\cong\C$, $e\in E$. We point out that $J_t$ is \emph{not} $\Om$-compatible. But, since $\begin{pmatrix}
 i & 0\\
 0 & j_t
\end{pmatrix}$ is tame (even compatible) and $B$ has compact support, the almost complex structure $J_t$ is $\Om$-tame for sufficiently small $B_t$. We denote by $\B^T(j)\subset\B(j)$ the non-empty open convex subset consisting of those $B\in\B(j)$ for which the corresponding $J_t$ is tame.

We use $J_t$ to introduce a bilinear form $\m$ on $T\big(\Lt_E\times \R\big)$ as follows. For $(\hat{u}_1,\hat{\eta}_1),(\hat{u}_2,\hat{\eta}_2)\in T_{([u,A],\eta)}\big(\Lt_E\times \R\big)=\Gamma(S^1,u^*TE)\times \R$ we set
\beq
\m\big((\hat{u}_1,\hat{\eta}_1),(\hat{u}_2,\hat{\eta}_2)\big):=-\int_0^1\Om\Big(J_t\big(u(t)\big)\hat{u}_1(t),\hat{u}_2(t)\Big)dt+\hat{\eta}_1\hat{\eta}_2\;.
\eeq
The bilinear form $\m$ is not symmetric but  positive definite  since $J$ is tame. Therefore we can define the vector field $\nabla \A_f^\tau(w)$ at $w=\big([u,A],\eta\big)$ implicitly by
\beq
d\A_f^\tau(w)\hat w=\m\big(\nabla \A_f^\tau(w),\hat{w}\big)\quad\forall\hat w\in T_w\big(\Lt_E\times \R\big)\;.
\eeq
An explicit expression is
\beq
\nabla \A_f^\tau(w)=
\begin{pmatrix}
-J_t(u)\big(\p_t u-\eta X_{\mu_\tau}(u)- X_F(u)\big)\\[1ex]
\displaystyle-\int_0^1\mu_\tau(u)dt
\end{pmatrix}\;.
\eeq
$\nabla\A_f^\tau$ is a gradient-like vector field for $\A_f^\tau$ since $J$ is tame and $B$ vanishes near critical points of $\A_f^\tau$: $B_t(e)=0$ for all $e\in\wp^{-1}(\U)$. Indeed, $d\A_f^\tau(w)\nabla \A_f^\tau(w)=\m\big(\nabla \A_f^\tau(w),\nabla \A_f^\tau(w)\big)\geq 0$ with equality if and only if $w\in\Crit(\A_f^\tau)$. Moreover, $\m$ is an inner product near critical points.

To construct Floer homology for $\A_f^\tau$ we study solutions $w=(u,\eta)\in C^\infty(\R\times S^1,E)\times C^\infty(\R,\R)$ to the Floer equations corresponding to \emph{positive} gradient flow of $\A_f^\tau$
\begin{equation}\label{eqn:Floer_eqn_for_A}
\left\{
\begin{aligned}
&\p_su+J_t(u)\big(\p_tu-\eta X_{\mu_\tau}(u) - X_F(u)\big)=0\\[1ex]
&\p_s\eta+\int_0^1\mu_\tau(u)dt=0\;.
\end{aligned}
\right. 
\end{equation}
Due to the assumption that $B$ vanishes near critical points the Floer equation thought of as a differential operator is Fredholm. The main ingredients for defining Floer homology are transversality and compactness for solution spaces of the Floer equation. This needs some attention in our framework due the restriction of the class of almost complex structures we consider and due to potential bubbling-off of holomorphic spheres.

The  projection $\wp$ maps critical points of the Rabinowitz Floer action functional $\A_f^\tau$ to those of the action functional of classical mechanics $\a_f$ on $(M,\om)$ 
\bea
\a_f:\Lt(M)&\pf\R\\
\a_f\big([q,\bar q]\big)&:=\int_{D^2}\bar q^*\om-\int_0^1f\big(q(t)\big)dt\;,
\eea
see Lemma \ref{lem:projection_of_crit_points_are_crit_points}. We recall that we chose the Morse function $f$ in a $C^2$-small fashion, see Convention \ref{convention_f_C2_small}. This implies that all critical points of $\a_f$ are critical points of $f$ with some capping, i.e.
\beq
\Crit(\a_f)\cong\Crit(f)\times \Gamma_M\;.
\eeq
We use the following convention for the Conley-Zehnder index for $(x,A)\in\Crit(\a_f)\cong\Crit(f)\times \Gamma_M$ 
\beq
\CZ^M(x,A)=-\Morse(x,f)+\tfrac12 \dim M+2c_1^{TM}(A)\;.
\eeq

\begin{Lemma}\label{lem:projection_of_crit_points_are_crit_points}
The projection $\wp$ induces the map
\bea
\Pi:\Crit(\A_f^\tau)&\pf\Crit(\a_f)\cong\Crit(f)\times \Gamma_M\\
\big([u,A],\eta\big)&\longmapsto [\wp(u),A]\;.
\eea 
\end{Lemma}
\begin{proof}
This follows directly from the definition \eqref{eqn:def_of_RFH_functional} of the action functional $\A_f^\tau$, see also Remark \ref{rmk:Gamma_E=Gamma_M}. 
\end{proof}
After a choice of $j\in\j$ the action functional $\a_f$ gives rise to the following Floer equation for $q:\R\times S^1\to M$
\beq\label{eqn:Floer_eqn_for_a}
\p_sq+j_t(q)\big(\p_tq - X_f(q)\big)=0\;.
\eeq
We recall that solutions of either Floer equation is of finite energy if and only if it converges at $\pm\infty$ to critical points of $\A_f^\tau$ resp.~$\a_f$. That is, a solution $w=(u,\eta)$ of the Floer equation \eqref{eqn:Floer_eqn_for_A} has finite energy
\beq
\int_\R\int_{S^1}\Big(|\p_su|^2+|\p_s\eta|^2\Big)dtds<\infty
\eeq
if and only if there exists $(u_\pm,\eta_\pm)\in\L(E)\times\R$ satisfying \eqref{eq:crit_eq} such that
\beq
\lim_{s\to\pm\infty}\big(u(s,\cdot),\eta(s)\big)=(u_\pm,\eta_\pm)
\eeq
and similarly for $\a_f$. Following the usual Morse-Bott ideas we denote for $w_\pm\in\c\subset\Crit(\A_f^\tau)$
\beq
\Mh(w_-,w_+):=\big\{w \text{ solves }\eqref{eqn:Floer_eqn_for_A}\text{ with}\lim_{s\to\pm\infty}w(s)\in W^{\pm}(w_\pm, h) \big\}
\eeq
the moduli space of finite energy solutions of the Floer equation of $\A_f^\tau$. Here $h:\Crit(\A_f^\tau)\to\R$ is the perfect Morse function from Definition \ref{def:index_of_critical_points} and $W^{+}(w_+, h)$ resp.~$W^{-}(w_-, h)$ denotes the stable resp.~unstable manifold of $h$ on $\Crit(\A_f^\tau)$.
Similarly, for $q_\pm\in\Crit(\a_f)$ let 
\beq
\Nh(q_-,q_+):=\big\{w \text{ solves }\eqref{eqn:Floer_eqn_for_a}\text{ with }\lim_{s\to\pm\infty}q(s)=q_\pm\big\}\;.
\eeq
Here, we abuse notation in the following sense. If $w_\pm=\big([u_\pm,\bar u_\pm],\eta_\pm\big)\in\Crit(\A_f^\tau)$ is given and $w=(u,\eta)$ is a finite energy solution of \eqref{eqn:Floer_eqn_for_A} then by $\lim_{s\to\pm\infty}w(s)=w_\pm$ we mean that
\beq
\lim_{s\to\pm\infty}\big(u(s,\cdot),\eta(s)\big)=(u_\pm,\eta_\pm)
\eeq
and
\beq\label{eq:Gamma_condition_for_grad_trajectories}
\big[(-\bar u_-)\#u\#\bar u_+\big]=0\in\Gamma_E\;.
\eeq
The same remark applies to $\a_f$. Unless $w_-=w_+$ the moduli space $\Mh(w_-,w_+)$ carries a free $\R$-action by shifts. We denote the quotient by 
\beq
\M(w_-,w_+):=\Mh(w_-,w_+)/\R
\eeq 
and similarly
\beq
\mathcal{N}(q_-,q_+):=\Nh(q_-,q_+)/\R\;.
\eeq 
All moduli spaces depend on additional data, e.g.~an almost complex structure, which we suppress in the notation.
\begin{Lemma}\label{lem:projection_of_grad_traj_are_grad_traj}
The projection $\wp$ induces the maps
\bea
\Pi:\Mh(w_-,w_+)&\pf \Nh\big(\Pi(w_-),\Pi(w_+)\big)\\
w=(u,\eta)&\longmapsto \Pi(w):=\wp(u)\;,\\[1ex]
\Pi:\M(w_-,w_+)&\pf \mathcal{N}\big(\Pi(w_-),\Pi(w_+)\big)\\
[w]&\longmapsto [\Pi(w)]\;.
\eea
\end{Lemma}

\begin{proof}
This follows immediately from the fact that $\wp_*(X_F)=X_f$, $\wp_*(X_{\mu_\tau})=0$ and the specific form of $J_t=\begin{pmatrix}
 i & B_t\\
 0 & j_t
\end{pmatrix}$, i.e.~$\wp_*\circ J=j\circ\wp_*$
\end{proof}

\subsection{Transversality}

We recall that $\j$ is the space of $S^1$-families of compatible almost complex structures on $(M,\om)$. We denote by
\beq
\j_{\mathrm{reg}}(f)\subset\j
\eeq 
the subset of $j\in\j$ with the following two properties.
\begin{itemize}
\item All finite energy solutions of the Floer equation for $\a_f$ with respect to $j$ are regular, i.e.~the operator obtained by linearizing the Floer equation is a surjective Fredholm operator for all finite energy solutions.
\item For every $t\in S^1$ all \emph{simple} $j_t$-holomorphic spheres are regular, i.e.~the operator obtained by linearizing the holomorphic sphere equation is a surjective Fredholm operator.
\end{itemize}
According to \cite{Floer_Hofer_Salamon_Transversality_in_elliptic_Morse_theory_for_the_symplectic_action} and \cite[Chapter 2]{McDuff_Salamon_J_holomorphic_curves_and_symplectic_topology} the subset $\j_{\mathrm{reg}}$ is of second category. For every $j\in\j_{\mathrm{reg}}(f)$ the moduli space $\Nh(q_-,q_+)$ is a smooth manifold of dimension
\beq\label{eq:dim_of_Nh}
\dim\Nh(q_-,q_+)=\CZ^M(q_+)-\CZ^M(q_-)\;.
\eeq
For $j\in\j_{\mathrm{reg}}(f)$ we denote by
\beq
\B_{\mathrm{reg}}(j)\subset\B(j)
\eeq
the subset of $B\in\B(j)$ with the following two properties
\begin{itemize}
\item All finite energy solutions of the Floer equation for $\A^\tau_f$ with respect to the corresponding $J$ are regular.
\item For every $t\in S^1$ all \emph{simple} $J_t$-holomorphic spheres are regular. 
\end{itemize}
We refer to  \cite{Abbo_Merry_Floer_homology_on_the_time_energy_extended_phase_space} for details on the linearization of the Rabinowitz Floer equations. For $B\in\B_{\mathrm{reg}}(j)$ the moduli space $\Mh(w_-,w_+)$ is a smooth manifold of dimension
\beq
\dim\Mh(w_-,w_+)=\mu(w_+)-\mu(w_-)\;.
\eeq
The next proposition shows that this class of almost complex structures is sufficiently large.
\begin{Prop}\label{prop:regular_J_is_generic}
 For all $j\in\j_{\mathrm{reg}}(f)$ the set $\B_{\mathrm{reg}}(j)\subset\B(j)$ is of second category.
\end{Prop}

\begin{proof}
We recall the splitting $TE\cong V\oplus H$, cf.~Remark \ref{rmk:splitting_TE=V+H}. Thus, we may consider the linearization of the Floer equation \eqref{eqn:Floer_eqn_for_A} in vertical resp.~horizontal directions $V$ resp.~$H$. Since the projection $\wp$ induces an isomorphism $\wp_*:(H,d\alpha)\to (TM,\om)$ and $j\in\j_{\mathrm{reg}}$, it follows from Lemma \ref{lem:projection_of_grad_traj_are_grad_traj} that the linearization is already surjective in horizontal directions. To show that it is for generic choice of $B$ also surjective in vertical directions we distinguish two cases for $w\in\Mh(w_-,w_+)$.\\[1ex]
\noindent\textit{Case 1. $\Pi(w)$ is non-constant}. We claim that $\Pi(w)$ necessarily leaves the neighborhood $\mathcal{U}$. We recall that $\mathcal{U}$ is the union of disjoint neighborhoods of all critical points of $f$ where each such neighborhood contracts onto a critical point, see the discussion before equation \eqref{eqn:conditions_of_B(J)}. If $\Pi(w)$ is contained in $\mathcal{U}$ then it has to be a gradient trajectory connecting the same critical point of $f$ with cappings $A$ and $A\#\Pi(w)$. Since $\Pi(w)$ is contained in $\mathcal{U}$ the two cappings are homotopic to each other: $A=A\#\Pi(w)\in\Gamma_M$. Thus, $\Pi(w)$ is a gradient trajectory from a critical point of $\a_f$ to itself (including cappings) and therefore $\Pi(w)$ is constant which is a contradiction. Therefore $\Pi(w)$ necessarily leaves the neighborhood $\mathcal{U}$. By \cite[Theorem
4.3]{Floer_Hofer_Salamon_Transversality_in_elliptic_Morse_theory_for_the_symplectic_action} the set of regular points for $\Pi(w)$ is open and dense.  Since $\Pi(w)$ leaves $\U$ we may apply Lemma \ref{lem:B_hats_for_transversality} below. Thus a standard argument, see for instance \cite[Section 5]{Floer_Hofer_Salamon_Transversality_in_elliptic_Morse_theory_for_the_symplectic_action} or \cite[Chapter 3]{McDuff_Salamon_J_holomorphic_curves_and_symplectic_topology} establishes that for generic $B
\in\B(j)$ the linearization of the gradient flow equation
is also vertically surjective. \\[1ex]
\noindent\textit{Case 2.  $\Pi(w)$ is constant.}
But then $w$ is a vortex and vortices are by \cite[Proposition A.1]{Albers_Frauenfelder_Bubbles_and_Onis} always transverse. In fact vortices are independent of the perturbation $B\in\B(j)$.\\[1ex]
It remains to prove that generically all \emph{simple} $J_t$-holomorphic spheres, $t\in S^1$, are regular. If we were not to restrict to upper triangular $J$ this is a standard result which relies on the fact that simple curves are somewhere injective, see \cite[Chapter 2]{McDuff_Salamon_J_holomorphic_curves_and_symplectic_topology} for details. 

We argue again as above. Due to the definition of $\j_\mathrm{reg}$ the linearization of a simple $J_t$-holomorphic sphere is already surjective in horizontal directions. For vertical directions we use the notion of somewhere horizontally injective points, see Definition \ref{def:horizontally_injective} below.
The important observation is that horizontally injective points still form a dense subset, see Lemma \ref{lem:simple_implies_horizontally_injective}. Therefore, we can apply again Lemma \ref{lem:B_hats_for_transversality} to conclude that for generic $B\in\B(j)$ all simple $J_t$-holomorphic curves are regular.
\end{proof}

\begin{Rmk}\label{rmk:definition_of_set_of_acs_we_ultimately_use}
We recall that $\B^T(j)\subset\B(j)$ denotes the non-empty open convex subset consisting of those $B\in\B(j)$ for which the corresponding $J$ is tame. From now on we always choose $j\in\j_{\mathrm{reg}}$ and $B\in\B^T_{\mathrm{reg}}(j):=\B_{\mathrm{reg}}(j)\cap \B^T(j)$.
\end{Rmk}

We recall the following notions and Lemmas considered in \cite{Frauenfelder_Habilitation,Albers_Frauenfelder_Bubbles_and_Onis}.

\begin{Def}\label{def:horizontally_injective}
A $J_t$-holomorphic curve $u:S^2 \to E$ is called \emph{somewhere horizontally injective} if there exists $z \in S^2$
such that
\beq
d^hu(z):=\wp_*\circ du(z) \neq 0, \quad u^{-1}(u(z))=\{z\}\;.
\eeq 
\end{Def}

It remains to prove the following two Lemmas.

\begin{Lemma}\label{lem:simple_implies_horizontally_injective}
Assume that $u:S^2 \to E$ is a simple $J_t$-holomorphic
curve. Then $u$ is horizontally injective on a dense set.
\end{Lemma}
\begin{proof}
We denote by $I(u)\subset S^2$ the subset of injective
points of $u$,  by $R(\wp(u)) \subset S^2$ the subset of
nonsingular points of $\wp(u)$ and  by $S(u)\subset S^2$ the subset of horizontally
injective points of $u$. Then
\begin{equation}\label{hi}
S(u)=I(u) \cap R(\wp(u)).
\end{equation}
We first observe that $\wp(u)$ is $j_t$-holomorphic. We claim that $\wp(u):S^2\to M$ is not constant, since otherwise $u$ would lie in one fiber and hence itself must be constant, contradicting the assumption that it is simple. Therefore, it follows from \cite[Lemma 2.4.1]{McDuff_Salamon_J_holomorphic_curves_and_symplectic_topology} that the complement of
$R(\wp(u))$ is finite. Moreover, it follows from \cite[Proposition
2.5.1]{McDuff_Salamon_J_holomorphic_curves_and_symplectic_topology} that the complement of $I(u)$ is countable. Hence by (\ref{hi}) the complement of $S(u)$ is countable. In particular, $S(u)$ is dense in $S^2$. 
\end{proof}

\begin{Lemma}\label{lem:B_hats_for_transversality}
We fix $e\in E\setminus\wp^{-1}(\U)$, $(v,h)\in T_eE=V_e\oplus H_e$ with $h\neq0$ and $t_0\in S^1$. Moreover, we fix $j\in\j$ and $B\in\B(j)$. Then there exist $\widehat B\in \Gamma_0(S^1\times E, \mathrm{L}(H,V))$ and $\widehat j\in T_j\j$ with
\beq\label{eqn:crazy_B_eqn}
\left\{
\begin{aligned}
&\widehat B(t_0,e)h=v\\
&i\widehat B+B\widehat j+\widehat B j=0\;.
\end{aligned}
\right. 
\end{equation}
\end{Lemma}

\begin{Rmk}
The second equation asserts that the pair $(\widehat j,\widehat B)$ corresponds to a tangent vector of the space of almost complex structures we are considering.
\end{Rmk}
   
\begin{proof}[Proof of Lemma \ref{lem:B_hats_for_transversality}]
First we extend $h$ resp.~$v$ to sections also denoted by $h$ resp.~$v$ supported in a small neighborhood of $e$. 
Then we define $\widehat j$  by
\beq
\widehat j h:=v,\quad \widehat j jh:=-jv, \quad\widehat j|_{\text{span}\{h,j h\}^\perp}:=0\;,
\eeq
where $\perp$ refers to the metric $\om(\cdot,j\cdot)$ on $H$. We point out that $\text{span}\{h,j h\}^\perp$ is $j$-invariant. Then $\widehat j$ satisfies the equation 
\beq
\widehat j j+j \widehat j=0
\eeq
that is, $\widehat j\in T_j\j$. Next we define $\widehat B$ by
\bea
\widehat B h&:=v\\
\widehat B jh&:=-(i\widehat B+B\widehat j)h\\
\widehat B|_{\text{span}\{h,j h\}^\perp}&:=0\;.
\eea
The first equation in \eqref{eqn:crazy_B_eqn} holds by construction. We show the second equation. We have
\bea
(i\widehat B+B\widehat j+\widehat B j)h=0
\eea
by the very definition of $\widehat B jh$. Moreover, 
\bea
(i\widehat B+B\widehat j+\widehat B j)jh&=i \widehat B jh + B\widehat j jh - \widehat Bh\\
&=-i (i\widehat B+B\widehat j)h - Bj\widehat jh - \widehat Bh\\
&=\widehat Bh - iB\widehat jh - Bj\widehat jh - \widehat Bh\\
&=- iB\widehat jh - Bj\widehat jh  \\
&=- (iB + Bj)\widehat jh  \\
&=0
\eea
where we used $jj=-1$, $\widehat B jh=-(i\widehat B+B\widehat j)h$, $\widehat j j+j \widehat j=0$, $ii=-1$ and finally $iB + Bj=0$, see \eqref{eqn:conditions_of_B(J)}. Finally, 
\bea
(i\widehat B+B\widehat j+\widehat B j)|_{\text{span}\{h,j h\}^\perp}=0
\eea
since $\text{span}\{h,j h\}^\perp$ is $j$-invariant and $\widehat B$ and $\widehat j$ vanish on it.
\end{proof}

This completes the discussion on transversality. 

\subsection{Compactness}\label{sec:compactness}

In this subsection we discuss the appropriate compactness results for the moduli spaces $\M(w_-,w_+)$ of unparametrized gradient flow trajectories. This follows the usual scheme of Rabinowitz Floer homology, that is, we need to establish the following for a sequence $(u_\nu,\eta_\nu)\in\M(w_-,w_+)$, $\nu\in\N$.

\begin{enumerate}\renewcommand{\theenumi}{\roman{enumi}}\itemsep=1ex
 \item A uniform $C^0$-bound for the loops $u_\nu$.
 \item A uniform $C^0$-bound on the Lagrange multipliers $\eta_\nu$.
 \item A uniform bound on the derivatives of the loops $u_\nu$.
\end{enumerate}
The first two are proved in \cite[Proposition 6.2 \& 6.4]{Frauenfelder_Habilitation}. We point out that the set-up in \cite{Frauenfelder_Habilitation} is the same as ours except for the following. Frauenfelder's assumption of $(E,\Om)$ being very negative is replaced by our assumption of semi-positivity. Moreover, the almost complex structures used are of the form $J=
\begin{pmatrix}
i&0\\0&j 
\end{pmatrix}$, i.e.~$B=0$. The uniform $C^0$-bound for the loops is based on a maximum principle which continues to hold since in our setting $B$ has compact support. The uniform $C^0$-bound on the Lagrange multipliers relies on a ``fundamental lemma'' which continues to hold verbatim. 

To prove a uniform bound on the derivatives of the loops we argue by contradiction, i.e.~by bubbling-off analysis. Indeed, since we already established uniform $C^0$-bounds for the loops and the Lagrange multipliers a blow-up of derivatives of $u_\nu$ leads to $J_t$-holomorphic spheres inside $E$. We claim that since we assume that $E$ is semi-positive we can apply the results of Hofer-Salamon \cite{Hofer_Salamon_Floer_homology_and_Novikov_rings} and ensure that for generic $S^1$-family almost complex structure $J$ of the form $J_t=
\begin{pmatrix}
i & B_t \\
0 & j_t
\end{pmatrix}
$ with $B_t\in\B^T_{\mathrm{reg}}(j)$, $t\in S^1$ the moduli spaces $\Mh(w_-,w_+)$ are compact up to breaking as long as $\CZ(w_+)-\CZ(w_-)\leq2$. In \cite{Hofer_Salamon_Floer_homology_and_Novikov_rings} Hofer-Salamon argue that bubbling-off of $J_t$-holomorphic spheres of Chern number at least 2 never occurs for index reasons. Moreover, they rule out bubbling-off of $J_t$-holomorphic sphere with Chern number less than 2 by carefully studying moduli spaces of $J_t$-holomorphic spheres. The crucial input is that for \emph{simple} holomorphic spheres the linearized operator is a surjective Fredholm operator, see \cite[Theorem 2.2]{Hofer_Salamon_Floer_homology_and_Novikov_rings}. We establishes the corresponding result for our restricted class of almost complex structures in Proposition \ref{prop:regular_J_is_generic}. Therefore, the results in \cite{Hofer_Salamon_Floer_homology_and_Novikov_rings} apply to the Floer equation for $\A_f^\tau$ and we conclude that the moduli spaces $\Mh(w_-,w_+)$ are compact up to breaking as long as $\mu(w_+)-\mu(w_-)\leq2$.

\subsection{Rabinowitz Floer homology}

We define Rabinowitz Floer homology with the help of Novikov rings. Alternative approaches are via mixed direct/inverse limits. How these relate has been studied in \cite{Cieliebak_Frauenfelder_Morse_homology_on_noncompact_manifolds}. The current approach is as in the original article \cite{Cieliebak_Frauenfelder_Restrictions_to_displaceable_exact_contact_embeddings}.

The spaces $\c_k$ and $\c$ of critical point of $\A_f^\tau$ were defined in Definition \ref{def:index_of_critical_points}. The vector space $\RFC_*(\A_f^\tau)$, graded by $\mu$ (see \eqref{eq:mu-index}),  is the set of all formal linear combinations 
\beq
\xi=\sum_{w\in \c} a_ww\;,\quad a_w\in\Z/2,
\eeq
subject to the Novikov condition
\beq\label{eqn:Novikov_condition}
\forall \kappa\in\R:\; \#\big\{w\in\c\mid a_w\neq0,\;\A_f^\tau(w)\geq \kappa\big\}<\infty\;.
\eeq
It is a module over the Novikov ring 
\beq
\Lambda_E:=\bigg\{\sum_{A\in\Gamma_E} n_A e^A\mid n_A\in\Z/2,\;\forall \kappa\in\R\colon \#\big\{A\in\Gamma_E\mid n_A\neq0,\;\Om(A)\geq \kappa\big\}<\infty\bigg\}\;.
\eeq
The multiplicative structure on $\Lambda_E$ is given by
\bean
\left(\sum_{A\in \Gamma_E}n_Ae^A\right)\cdot\left(\sum_{B\in \Gamma_E}m_Be^B\right)
                &:=\sum_A\sum_B(n_A\cdot m_B)\,e^{A+B}=\sum_C\Big(\sum_An_A\cdot m_{C-A}\Big)e^{C}
\eea
and the action of $\Lambda_E$ on $\RFC_k(\A_f^\tau)$ by
\beq
\left(\sum_{A\in\Gamma_E} n_A e^A\right)\cdot\left(\sum_{w\in \c} a_ww\right):=\sum_w \bigg(\sum_A n_A\cdot a_{w\#-A}\bigg)w\;,
\eeq
where we use the following notation. If $w=([u,B],\eta)$ then $w\#-A=([u,B-A],\eta)$. The differential $\p$ on $\RFC_*(\A_f^\tau)$ is defined by
\bea
\p:\RFC_k(\A_f^\tau)&\pf \RFC_{k-1}(\A_f^\tau)\\
\p\, w&:=\sum_{z\in\c_{k-1}} \#_2\M(z,w)\,z\;.
\eea
The compactness results described in section \ref{sec:compactness} imply that $\M(z,w)$ is a finite set and $ \#_2\M(z,w)\in\Z/2$ denotes its parity. Moreover, compactness up to breaking implies $\p\circ\p=0$. The Rabinowitz Floer homology is then defined by
\beq
\RFH_k(\A_f^\tau):=\H_k\big(\RFC_*(\A_f^\tau),\p\big),\quad k\in\tfrac{1}{2}+\Z\;.
\eeq

\begin{Rmk}
To define $\RFH_*(\A_f^\tau)$ we made auxiliary choices, notably $\tau$ and $f$. The assumption that $f$ is $C^2$-small is not necessary for defining $\RFH_*(\A_f^\tau)$, see \cite{Albers_Frauenfelder_Bubbles_and_Onis} for more details. Nevertheless, we decided to make this assumption throughout this article. The choices of $\tau$ and $f$ become relevant in the proof of Theorem \ref{thm:main}. The methods of \cite{Cieliebak_Frauenfelder_Restrictions_to_displaceable_exact_contact_embeddings} show that $\RFH_*(\A_f^\tau)$ is independent of all these choices.
\end{Rmk}

\begin{Rmk}
We recall that we restrict ourselves to  the class of almost complex structures $J$ of the form $J=
\begin{pmatrix}
i & B \\
0 & j
\end{pmatrix}
$ with $B\in\B^T_{\mathrm{reg}}(j)$. It is unclear to us whether it is possible to extend the definition of $\RFH_*(\A_f^\tau)$ beyond this class of almost complex structures. We crucially rely on Frauenfelder's result, namely that the fact that the projection of the Floer equation of $\A_f^\tau$ gives the Floer equation of $\a_f$ on $M$ can be used to obtain uniform $C^0$-bounds for the Lagrange multiplier. For this $\wp$ needs to be $J$-$j$-holomorphic. 
\end{Rmk}

\section{A filtration and the proof of vanishing}

We use the fact that $\RFC_*(\A_f^\tau)$ admits a filtration. For $l\in\Z$ we set
\beq
\RFC_k^l(\A_f^\tau):=\left\{\sum_{w} a_ww\in\RFC_k(\A_f^\tau)\mid \CZ^M\big(\Pi(w)\big)= l\right\}
\eeq
and 
\beq
\RFC_k^{\leq l}(\A_f^\tau):=\left\{\sum_w a_ww\in\RFC_k(\A_f^\tau)\mid \CZ^M\big(\Pi(w)\big)\leq l\right\}\;,
\eeq
where we recall that $A\in\Gamma_E\cong\Gamma_M$, see Remark \ref{rmk:Gamma_E=Gamma_M}.

\begin{Lemma}\label{lem:properties_of_p=sum_p_i}
\beq
\p \Big(\RFC_k^{\leq l}(\A_f^\tau)\Big)\subset \RFC_{k-1}^{\leq l}(\A_f^\tau) 
\eeq
hence we can decompose 
\beq
\p=\sum_{i\geq0} \p_i=\p_0+\p_1+\cdots
\eeq
with
\bea
\p_i:\RFC_k^l(\A_f^\tau)&\pf \RFC_{k-1}^{l-i}(\A_f^\tau)\\
\p_iw&:=\!\!\!\!\!\!\!\!\sum_{\substack{z\in\c_{k-1}\\\CZ^M(\Pi(z))=\CZ^M(\Pi(w))-i}}\!\!\!\!\!\!\!\! \#_2\M(z,w)z\;.
\eea
\end{Lemma}
\begin{proof}
This is a direct consequence of Lemma \ref{lem:projection_of_grad_traj_are_grad_traj} together with \eqref{eq:dim_of_Nh}.
\end{proof}

\begin{Rmk}\label{rmk:p_i_drops_upper_degree_by_i}
From $\p^2=0$ and the filtration we derive for every $i\geq0$ the equation
\beq\label{eqn:magic_p_i_p_j}
\sum_{j=0}^{i}\p_j\p_{i-j}=0\;.
\eeq
E.g.~$\p_0\p_0=0$, $\p_0\p_1+\p_1\p_0=0$ etc. In particular, $\p_0$ is a differential.
\end{Rmk}

The main idea for proving Theorem \ref{thm:main} is that $\p_0$ counts solutions of the Floer equation \eqref{eqn:Floer_eqn_for_A} which are entirely contained inside fibers of $\wp$ over critical points of $f$. Thus the homology of $\p_0$ is the sum of $\Crit(f)\times\Gamma_M$-many copies of the Rabinowitz Floer homology of $\big(\Sigma_\tau\cap\wp^{-1}(q),\wp^{-1}(q)\big)\cong(S^1,\C)$, $q\in\Crit(f)$, each of which vanishes. 

\begin{Prop}\label{prop:p_0_counts_differentials_in_a_fiber}
The differential $\p_0$ counts precisely the solutions $w=(u,\eta)$ of the Floer equation \eqref{eqn:Floer_eqn_for_A} with image contained entirely in a fiber over some critical point of $f$. That is, there exists $q\in\Crit(f)$ such that $u(\R\times S^1)\subset \wp^{-1}(q)$. Moreover, if $w_\pm=([u_\pm,A_\pm],\eta_\pm)\in\Crit(\A_f^\tau)$ are the asymptotic limits of $w$ then
\beq
A_-=A_+\in\Gamma_E\;.
\eeq
\end{Prop}

\begin{proof}
Let  $w=(u,\eta)$ be a gradient flow line from $w_-=([u_-,A_-],\eta_-)$ to $w_+=([u_+,A_+],\eta_+)$ with
\beq\label{eqn:cz=cz}
\CZ^M\big(\Pi(w_+)\big)=\CZ^M\big(\Pi(w_-)\big)\;.
\eeq
Using Lemma \ref{lem:projection_of_grad_traj_are_grad_traj} we see that 
\beq
\wp(u)\in\widehat{\mathcal{N}}\big(\Pi(w_+),\Pi(w_-)\big)\;.
\eeq
According to \eqref{eq:dim_of_Nh}, equation \eqref{eqn:cz=cz} implies that that 
\beq
\dim\widehat{\mathcal{N}}\big(\Pi(w_+),\Pi(w_-)\big)=0\;,
\eeq
which in turn implies that $\wp(u)$ is $s$-independent, i.e.~constant $\wp(u)=q\in\Crit(f)$, see Lemma \ref{lem:projection_of_crit_points_are_crit_points}. In other words, $u(\R\times S^1)\subset \wp^{-1}(q)$. Moreover, in view of \eqref{eq:Gamma_condition_for_grad_trajectories}, we have 
\beq
A_-=A_+\in\Gamma_M\cong\Gamma_E\;.
\eeq
This finishes the proof.
\end{proof}

\begin{Cor}\label{cor:p_0_homology=0}
\beq
\H_k\big(\RFC_*(\A_f^\tau),\p_0\big)=0\quad \forall k\in\tfrac{1}{2}+\Z\;.
\eeq 
\end{Cor}

\begin{proof}
For $q\in\Crit (f)$ we fix an identification 
\beq
\big(\Sigma_\tau\cap\wp^{-1}(q),\wp^{-1}(q)\big)\cong(S^1_\tau,\C)
\eeq
together with the symplectic form, its primitive and the complex structure $i$. Here $S^1_\tau$ is the circle bounding a disk of area $\pi\tau^2$. For $A\in\Gamma_M$ we denote by 
\beq
\RFC_*(q,A)
\eeq
the vector space generated over $\Z/2$ by critical points of the form $\big([u,A]^\pm,\eta\big)\in \c\subset\Crit(\A_f^\tau)$ with $\wp(u)=q$. Proposition \ref{prop:p_0_counts_differentials_in_a_fiber} implies that $\RFC_*(q,A)$ is a $\p_0$-subcomplex of $\RFC_*(\A_f^\tau)$. With the above identification we see that
\beq
\big(\RFC_{k+2c_1^{TE}(A)+\frac{1}{2}\dim M}(q,A),\p_0\big)=\big(\RFC_k(S^1,\C),\p\big)\;.
\eeq
Let $v$ be the primitive Reeb orbit over $q$ then all generators are of the form $\big([v^n,A]^\pm, \eta=n-f(q)\big)$. Since $\RFH_*(S^1,\C)=0$ due to  \cite{Cieliebak_Frauenfelder_Restrictions_to_displaceable_exact_contact_embeddings,Albers_Frauenfelder_Leafwise_intersections_and_RFH}  and $\mu\big([v^{n-1},A]^\pm, n-1-f(q)\big)+2=\mu\big([v^n,A]^\pm, n-f(q)\big)$ from Lemma \ref{lemma:index_of_fiber}, we know that 
\bea\label{eq:differential_p_0}
\p_0\big([v^n,A]^-, n-f(q)\big)&=\big([v^{n-1},A]^+,n-1-f(q)\big)\\
\p_0\big([v^n,A]^+,n-f(q)\big)&=0\;.
\eea
Let $\xi=\sum_w a_ww\in\RFC_k(\A_f^\tau)$ with $\p_0\xi=0$, i.e.~$\sum_w a_w\p_0w=0$. If $a_w\neq0$ then $w$ is of the form $\big([v^n,A]^+,n-f(q)\big)$ and we let $w'$ be the corresponding element $\big([v^{n+1},A]^-,n+1-f(q)\big)$. Then
\beq
\xi':=\sum_w a_w w'
\eeq
satisfies the Novikov condition, i.e.~$\xi'\in\RFC_{k+1}(\A_f^\tau)$, since $\A_f^\tau(w')=\A_f^\tau(w)+\tau$ due to \eqref{eq:computation_of_action}. From \eqref{eq:differential_p_0}, $\p_0\xi'=\xi$ and this completes the proof.
\end{proof}

\begin{Lemma}\label{lem:c_great_equal_0_implies_CZ_downstairs_bounded_from_above} 
We assume now that $c_1^{TM}=c\om:\pi_2(M)\to\Z$. 
\begin{itemize}
\item In case $c=0$ we have for all $\tau>0$ 
\beq
\RFC_k(\A_f^\tau)=\bigoplus_{l=-\frac12\dim M}^{\frac12\dim M}\RFC_k^{l}(\A_f^\tau)
\eeq 
and
\beq
\p_n=0\quad\forall n\geq\dim M+1\;.
\eeq
\item In case $c\geq1$ we assume $(c-1)\tau<1$. Then a formal sum $\xi=\sum_w a_w w$, $a_w\in\Z/2$, $w\in\c_k$ satisfies the Novikov condition
\beq
\forall \kappa\in\R:\; \#\big\{w\in\c_k\mid a_w\neq0,\;\A_f^\tau(w)\geq \kappa\big\}<\infty
\eeq
if and only if 
\beq
\forall \kappa\in\R:\; \#\big\{w\in\c_k\mid a_w\neq0,\;\CZ^M\big(\wp(w)\big)\geq \kappa\big\}<\infty\;.
\eeq 
In particular, for all $\xi\in\RFC_k(\A_f^\tau)$ there exists $l(\xi)\in\Z$ with
\beq\label{eqn:xi_in_l(xi)_part}
\xi\in\RFC_k^{\leq l(\xi)}(\A_f^\tau)\;.
\eeq
\end{itemize}
\end{Lemma}

\begin{proof}
If we write $w=\big([v^{n},A]^\pm,\eta\big)\in\c_k$ then according to \eqref{eq:computation_of_action} and Definition \ref{def:index_of_critical_points}
\bea
\A_f^\tau(w)&=n\tau+\om(A)-(\tau+1)f(\wp(v))\\[1ex]
k=\mu(w)&=2n+2c_1^{TE}(A)-\Morse(\wp(v);f)+\tfrac12 \dim M\big(\pm\tfrac12 )\\
&=2n+2(c-1)\om(A)-\Morse(\wp(v);f)+\tfrac12 \dim M\big(\pm\tfrac12 )\;.
\eea
We solve the second equation for $n$
\bea
n&=\tfrac12 k-(c-1)\om(A)+\tfrac12 \big[\Morse(\wp(v);f)-\tfrac12 \dim M\big(\pm\tfrac12 )\big]
\eea
and abbreviate $e:=\Morse(\wp(v);f)-\tfrac12 \dim M\big(\pm\tfrac12 )$. In particular, $|e|\leq\tfrac12 \dim M+\tfrac12 $. Thus, we can rewrite the action value as
\bea\label{eqn:formula_A}
\A_f^\tau(w)&=n\tau+\om(A)-(\tau+1)f(\wp(v))\\
&=\big(\tfrac12 k-(c-1)\om(A)+\tfrac12 e\big)\tau+\om(A)-(\tau+1)f(\wp(v))\\
&=\big(1-(c-1)\tau\big)\om(A)+\tfrac12 \big(k+e\big)\tau-(\tau+1)f(\wp(v))\;.
\eea
Next we observe that
\bea\label{eqn:formula_CZ}
\CZ^M(\Pi(w))&=-\Morse(\wp(v);f)+\tfrac12 \dim M+2c_1^{TM}(A)\\
&=-\Morse(\wp(v);f)+\tfrac12 \dim M+2c\om(A)\;.
\eea
In case $c\geq 1$ and $(c-1)\tau<1$ equations \eqref{eqn:formula_A} and \eqref{eqn:formula_CZ} imply the if-and-only-if statement of the Lemma. The statement \eqref{eqn:xi_in_l(xi)_part} follows from the if-part of the if-and-only-if statement since $\xi\in\RFC_k(\A_f^\tau)$ satisfies the Novikov condition by the very definition of $\RFC$.

The case $c=0$ follows immediately from equation  \eqref{eqn:formula_CZ}.
\end{proof}

We are now in the position to prove Theorem \ref{thm:main}. We treat the symplectically aspherical case last and assume now that $c_1^{TM}=c\om$. We first consider the case $c\geq0$. If $c=0$ we assume that $(E,\Om)$ is semi-positive.

\begin{proof}[Proof of Theorem \ref{thm:main} for $c\geq0$]
We fix $\xi\in\RFC_k(\A_f^\tau)$ with
\beq
\p \xi=0\;.
\eeq
Our aim is to construct $\theta\in\RFC_{k+1}(\A_f^\tau)$ with $\p\theta=\xi$. We split $\xi$ as follows.
\beq
\xi=\sum_{l=-\infty}^{l(\xi)} \xi_l \quad\text{with}\quad\xi_l\in\RFC_k^l(\A_f^\tau) \;,
\eeq 
where $l(\xi)\in\Z$ is taken from Lemma \ref{lem:c_great_equal_0_implies_CZ_downstairs_bounded_from_above}. If $c=0$ we set $l(\xi):=\tfrac12 \dim M$. We expand $\p\xi=0$ according to $\p=\sum_{i\geq0}\p_i$ and collect terms in $\RFC_{k-1}^{l(\xi)-I}(\A_f^\tau)$ for all $I\geq0$. We recall from Lemma \ref{lem:properties_of_p=sum_p_i} that $\p_i$ drops the upper degree by $i$. This leads to 
\beq\label{eqn:p_xi=0_in_long}
\sum_{i=0}^{I}\p_i\xi_{l(\xi)+i-I}=0\;,
\eeq
since $\p_i\xi_m\in\RFC_{k-1}^{l(\xi)-I}(\A^\tau_f)$ if and only if $m-i=l(\xi)-I$.\\[1ex]
\begin{Claim}\label{claim1}
 For all $l\leq l(\xi)$ there exists $\theta_l\in\RFC_{k+1}^l(\A_f^\tau)$ such that 
\beq\label{eqn:inductive_equation}
\sum_{i=0}^{I}\p_i\theta_{l(\xi)+i-I}=\xi_{l(\xi)-I}
\eeq 
holds for $I\geq0$.
 \end{Claim}
\begin{proof}[Proof of Claim \ref{claim1}]
We inductively construct $\theta_l$. For $I=0$ equation \eqref{eqn:p_xi=0_in_long} reduces to 
\beq
\p_0 \xi_{l(\xi)}=0\;.
\eeq
Corollary \ref{cor:p_0_homology=0} implies that there exists $\theta_{l(\xi)}\in\RFC_{k+1}^{l(\xi)}(\A_f^\tau)$ with
\beq
\p_0\theta_{l(\xi)}=\xi_{l(\xi)}\;.
\eeq
Now assume that we already constructed $\theta_{l(\xi)},\ldots, \theta_{l(\xi)-(I-1)}$ satisfying equation \eqref{eqn:inductive_equation}. Then we compute 
\bea
\p_0\bigg(\xi_{l(\xi)-I}-\sum_{i=1}^I\p_i\theta_{l(\xi)+i-I}\bigg)&\stackrel{\phantom{\ast(\ast)\ast}}{=}\p_0\xi_{l(\xi)-I}-\sum_{i=1}^I\p_0\p_i\theta_{l(\xi)+i-I}\\
&\stackrel{\phantom{\ast}(\ast)\phantom{\ast}}{=}\p_0\xi_{l(\xi)-I}+\sum_{i=1}^I\sum_{j=1}^i\p_j\p_{i-j}\theta_{l(\xi)+i-I}\\
&\,\stackrel{(\ast\ast)}{=}\p_0\xi_{l(\xi)-I}+\sum_{j=1}^I\p_j\left(\sum_{i=j}^I\p_{i-j}\theta_{l(\xi)+i-I}\right)\\
&\stackrel{\phantom{\ast(\ast)\ast}}{=}\p_0\xi_{l(\xi)-I}+\sum_{j=1}^I\p_j\left(\sum_{i=0}^{I-j}\p_{i}\theta_{l(\xi)+i-(I-j)}\right)\\
&\stackrel{(\ast\ast\ast)}{=}\p_0\xi_{l(\xi)-I}+\sum_{j=1}^I\p_j\xi_{l(\xi)+j-I}\\
&\stackrel{\phantom{\ast(\ast)\ast}}{=}\sum_{j=0}^I\p_j\xi_{l(\xi)+j-I}\\
&\stackrel{\phantom{\ast(\ast)\ast}}{=}0\;.
\eea
Here we used equation \eqref{eqn:magic_p_i_p_j} in $(\ast)$, the usual relabeling $\displaystyle\sum_{i=1}^I\sum_{j=1}^i=\sum_{j=1}^I\sum_{i=j}^I$ in $(\ast\ast)$, the induction hypothesis \eqref{eqn:inductive_equation} in $(\ast\!\ast\!\ast)$ and \eqref{eqn:p_xi=0_in_long} at the end. Using again Corollary \ref{cor:p_0_homology=0} we find $\theta_{l(\xi)-I}\in\RFC_{k+1}^{l(\xi)-I}(\A_f^\tau)$ with
\beq
\p_0\theta_{l(\xi)-I}=\xi_{l(\xi)-I}-\sum_{i=1}^I\p_i\theta_{l(\xi)+i-I}\;,
\eeq
in other words
\beq\label{eqn:p_theta=xi_in_parts}
\sum_{i=0}^I\p_i\theta_{l(\xi)+i-I}=\xi_{l(\xi)-I}\;.
\eeq
This proves claim \ref{claim1}. 
\end{proof}
Now we consider
\beq
\theta:=\sum_{l=-\infty}^{l(\xi)}\theta_l\;.
\eeq
We will show that $\theta$ satisfies the Novikov condition and $\p\theta=\xi$. We need slightly different arguments for the cases $c=0$ and $c\geq1$.

If $c=0$ then $\xi=\sum_{l=-\frac12\dim M}^{\frac12\dim M} \xi_l$, see Lemma \ref{lem:c_great_equal_0_implies_CZ_downstairs_bounded_from_above}, is the sum of finitely many non-zero $\xi_l$ each of which satisfies the Novikov condition. Since $\p=\p_0+\cdots\p_{\dim M}$ we obtain  only finitely many non-zero $\theta_l$ each of which satisfies the Novikov condition by (the proof of) Corollary \ref{cor:p_0_homology=0}. Thus, $\theta=\sum\theta_l$ satisfies the Novikov condition, too.

If $c\geq1$ then Lemma \ref{lem:c_great_equal_0_implies_CZ_downstairs_bounded_from_above} implies that each $\theta_i\in\RFC_{k+1}^i(\A_f^\tau)$ is a finite sum of elements in $\c$ (as opposed to a general Novikov sum.)  Thus, using again Lemma \ref{lem:c_great_equal_0_implies_CZ_downstairs_bounded_from_above} we see that $\theta=\sum_{l=-\infty}^{l(\xi)}\theta_l$ satisfies the Novikov condition.

In both cases the equation $\p\theta=\xi$ holds by construction. Indeed, the part of 
\beq
\p\theta=\sum_{l=-\infty}^{l(\xi)}\p\theta_l=\sum_{l=-\infty}^{l(\xi)}\sum_{i=0}^\infty\underbrace{\p_i\theta_l}_{\in\RFC_k^{l-i}}\in\RFC_k(\A_f^\tau)
\eeq
in $\RFC_k^r(\A_f^\tau)$ is $\displaystyle\sum_{i=0}^{l(\xi)-r}\p_i\theta_{i+r}$. By relabeling $I=l(\xi)-r$ we compute
\bea
\p\theta&=\sum_{r=-\infty}^{l(\xi)}\sum_{i=0}^{l(\xi)-r}\p_i\theta_{i+r}\\
&=\sum_{I=0}^\infty\sum_{i=0}^{I}\p_i\theta_{l(\xi)+i-I}\\
&=\sum_{I=0}^\infty\xi_{l(\xi)-I}\\
&=\sum_{l=-\infty}^{l(\xi)}\xi_l\\
&=\xi\;,
\eea
where we used equation \eqref{eqn:p_theta=xi_in_parts} in the third equality. Thus, for every $\xi\in\RFC_k(\A^\tau_f)$ with $\p\xi=0$ we constructed $\theta\in\RFC_{k+1}(\A^\tau_f)$ with $\p\theta=\xi$. This finishes the proof.
\end{proof}

\begin{proof}[Proof of Theorem \ref{thm:main} for $2c\nu\leq-\dim M$]
In this proof we make the assumption that the Morse function $f:M\to\R$ additionally satisfies $f(M)\subset(0,1)$. We fix $\xi\in\RFC_k(\A_f^\tau)$ with
\beq
\p \xi=0\;.
\eeq
We will again construct $\theta\in\RFC_{k+1}(\A_f^\tau)$ with $\p\theta=\xi$. This time we split $\xi$ as follows.
\beq
\xi=\sum_{A\in\Gamma_M}\xi_A\quad\text{with}\quad \xi_A=\sum_{\substack{w\in\c_k \\ [\Pi(w)]\in\Crit (f)\times\{A\}}}a_ww\;.
\eeq 
\begin{Claim}\label{claim2}
If $\M\Big(\big([v,B],\hat\eta\big),\big([u,A],\eta\big)\Big)\neq\emptyset$ then $A=B\in\Gamma_M\cong\Gamma_E$. 
\end{Claim}

\begin{proof}[Proof of Claim \ref{claim2}]
We compare action and Conley-Zehnder index of $\Pi\big([u,A],\eta\big)$ and $\Pi\big([v,B],\hat\eta\big)$. Using that the moduli space is non-empty we conclude
\beq
\a_f\big(\Pi\big([u,A],\eta\big)\big)=\om(A)-f(\wp(u))\geq\om(B)-f(\wp(v))=\a_f\big(\Pi\big([v,B],\hat\eta\big)\big)
\eeq
and
\bea\label{eqn:ineq_CZ_claim}
\CZ^M\big(\Pi\big([u,A],\eta\big)\big)&=-\Morse\big(\wp(u),f\big)+\tfrac12\dim M+2c_1^{TM}(A)\\
&\geq-\Morse\big(\wp(v),f\big)+\tfrac12\dim M+2c_1^{TM}(B)\\
&=\CZ^M\big(\Pi\big([v,B],\hat\eta\big)\big)\;.
\eea
We assume now that $A\neq B$. We recall that $f(M)\subset(0,1)$. Thus, the first inequality simplifies to
\beq
\om(A)>\om(B)
\eeq
since $\om\big(\pi_2(M)\big)=\nu\Z$. From $c_1^{TM}=c\om$ with $c<0$ we conclude then $c_1^{TM}(A)<c_1^{TM}(B)$. The minimal Chern number of $M$ equals $-c\nu$. Thus, we have
\beq\label{eqn:inequ_in_claim}
c_1^{TM}(A)\leq c_1^{TM}(B)+c\nu
\eeq
and from \eqref{eqn:ineq_CZ_claim}
\bea
2c_1^{TM}(A)&\geq-\Morse\big(\wp(v),f\big)+\Morse\big(\wp(u),f\big)+2c_1^{TM}(B)\\
&\geq-\dim M+2c_1^{TM}(B)\;.
\eea
We conclude that
\beq
2c\nu\geq-\dim M\;.
\eeq
In case $2c\nu<-\dim M$ we arrive at a contradiction. It remains to treat the case $2c\nu=-\dim M$. In this case we claim that the inequality \eqref{eqn:inequ_in_claim} necessarily becomes the equality
\beq
c_1^{TM}(A)=c_1^{TM}(B)+c\nu\;.
\eeq
Otherwise \eqref{eqn:inequ_in_claim} is actually of the form $c_1^{TM}(A)\leq c_1^{TM}(B)+2 c\nu$ since $-c\nu$ is the minimal Chern number of $M$. As above this implies then that $4cv\geq-\dim M$, i.e.~$2\dim M\leq\dim M$, and thus $\dim M=0$. I.e.~we are left with the case $\Sigma=S^1\subset\C=E$ in which Theorem \ref{thm:main} is true: $\RFH_*(S^1,\C)=0$,  \cite{Cieliebak_Frauenfelder_Restrictions_to_displaceable_exact_contact_embeddings,Albers_Frauenfelder_Leafwise_intersections_and_RFH}

We combine $c_1^{TM}(A)=c_1^{TM}(B)+c\nu$ with \eqref{eqn:ineq_CZ_claim} and arrive at
\beq
\Morse\big(\wp(v),f\big)-\Morse\big(\wp(u),f\big)\geq-2c\nu=\dim M
\eeq
which turns the inequality \eqref{eqn:ineq_CZ_claim} into an equality:
\beq
\CZ^M\big(\Pi\big([u,A],\eta\big)\big)=\CZ^M\big(\Pi\big([v,B],\hat\eta\big)\big)\;.
\eeq
Now we proceed as in the proof of Proposition \ref{prop:p_0_counts_differentials_in_a_fiber} in order to conclude that all element in $\M\Big(\big([v,B],\hat\eta\big),\big([u,A],\eta\big)\Big)$ are actually differentials which are entirely contained in fibers of $E$ and thus $A=B$, again by Proposition \ref{prop:p_0_counts_differentials_in_a_fiber}.
\end{proof}
We recall that we split the cycle $\xi$ as 
\beq
\xi=\sum_{A\in\Gamma_M}\xi_A\;.
\eeq 
It follows from the claim \ref{claim2} and $\p\xi=0$ that
\beq
\p \xi_A=0\quad\forall A\in\Gamma_M\;.
\eeq
Observe that for every $A$ the sum 
\beq
\xi_A=\sum_{\substack{w\in\c_k \\ [\Pi(w)]\in\Crit (f)\times\{A\}}}a_ww
\eeq
is finite. Indeed, we know that $w$ is of the form $w=([u,A],\eta)$ with fixed $A$ and $u$ being an $l$-fold cover of a simple Reeb orbit over a critical point of $f$. Moreover, the index of $w$ is fixed: $\mu(w)=k$. Therefore, Definition \ref{def:index_of_critical_points} of the index $\mu$ and the index formula Lemma \ref{lemma:index_of_fiber} allow only for finitely many combinations. The number of possibilities is bounded by $\tfrac12\dim M$. In particular, the number of possibilities does \emph{not} depend on $A$.

Now, we apply again the inductive procedure \eqref{eqn:inductive_equation} from the proof in case $c\geq0$ to obtain $\theta_A\in\RFC_{k+1}(\A_f^\tau)$ with
\beq
\p\theta_A=\xi_A\;.
\eeq
If we set
\beq
\theta:=\sum_{A\in\Gamma_M}\theta_A
\eeq
then claim \ref{claim2} implies $\p\theta=\xi$. As above it remains to check that $\theta$ satisfies the Novikov condition \eqref{eqn:Novikov_condition}. For this we express for some $A$
\beq
\theta_A=\sum_{\substack{z\in\c_{k+1} \\ [\Pi(z)]\in\Crit (f)\times\{A\}}}b_zz\;.
\eeq
The same argument we used to conclude that each $\xi_A$ is a finite sum gives the same for $\theta_A$. Moreover, since $\mu(\theta_A)=k+1$, $\p\theta_A=\xi_A$ again the claim, the index formula Lemma \ref{lemma:index_of_fiber} and the computation of the action \eqref{eq:computation_of_action} implies that exists $C>0$ such that \beq
|\A_f^\tau(z)-\A_f^\tau(w)|\leq C
\eeq
whenever $\M(w,z)\neq\emptyset$ for some $w$ appearing in $\xi_A$ and $z$ in $\theta_A$. The constant $C$ does \emph{not} depend on $A$, indeed we may choose $C=\tfrac\tau2\dim M+\max f-\min f$.

Thus, $\xi$ satisfying the Novikov condition implies that $\theta$ satisfies the Novikov condition since their actions are of bounded distance. This completes the proof.
\end{proof}

\begin{proof}[Proof of Theorem \ref{thm:main} for $\om\big(\pi_2(M)\big)=0$]
We follow the proof of the case $2c\nu\leq-\dim M$. We first establish claim \ref{claim2}, i.e.~that $\M\big(\big([v,B],\hat\eta\big),\big([u,A],\eta\big)\big)\neq\emptyset$ implies $A=B\in\Gamma_M$ holds without assuming $c_1^{TM}=c\om$ under the assumption that the Morse function $f:M\to\R$ is sufficiently small.
 
We assume otherwise. Then we find a sequence $\epsilon_n\to 0$ and a sequence of elements $w_n\in\M\big(\big([v_n,B_n],\hat\eta_n\big),\big([u_n,A_n],\eta_n\big)\big)$ where $w_n$ satisfies the Floer equation for $\A_{f_n}^\tau$ with $f_n:=\epsilon_nf$. By definition of $\M$ we have $c_1^{TE}(-A_n\#w_n\#B_n)=0$. Since $\om\big(\pi_2(M)\big)=0$ we can identify $c_1^{TE}=c_1^{TM}:\pi_2(E)\cong\pi_2(M)\to\Z$. We recall that $\Pi(w_n)$ are solutions of the Floer equation of $\a_{f_n}$ with 
\beq
c_1^{TM}(-A_n\#\Pi(w_n)\#B_n)=c_1^{TE}(-A_n\#w_n\#B_n)=0\;.
\eeq
We point out that the Floer cylinders $\Pi(w_n)$ topologically form spheres since their asymptotic limits lie in $\Crit(f_n)$. Therefore, we can rewrite
\beq
c_1^{TM}(-A_n\#\Pi(w_n)\#B_n)=-c_1^{TM}(A_n)+c_1^{TM}(\Pi(w_n))+c_1^{TM}(B_n)=0\;.
\eeq
Moreover, they have uniformly bounded energy since their energy is given by the action difference of $\a_{f_n}$ which, in turn, is bounded by $\max f_n-\min f_n$ thanks to our assumption $\om\big(\pi_2(M)\big)=0$. Therefore, we can take the Floer-Gromov limit of $\Pi(w_n)$. Floer-Gromov compactness implies that we find a bubble tree of holomorphic spheres in $(M,\om)$. Since we assume that $\om\big(\pi_2(M)\big)=0$ all holomorphic spheres are constant and therefore
\beq
c_1^{TM}(\Pi(w_n))=0
\eeq
for sufficiently large $n$ from which we conclude
\beq
c_1^{TM}(A_n)=c_1^{TM}(B_n)
\eeq
for sufficiently large $n$. That is, the above claim indeed holds for sufficiently small $f:M\to\R$. We now can proceed as in the proof of the case $2c\nu\leq-\dim M$. 
\end{proof}

\begin{Rmk}
In the latter two cases of the proof of Theorem \ref{thm:main} we assume that the auxiliary Morse function $f$ is very small. This is an echo of the `true' proof of Theorem \ref{thm:main} in the full Morse-Bott setting, i.e.~the case of $\A_{f=0}^\tau$. Indeed, in both cases $2c\nu\leq-\dim M$ and  $\om\big(\pi_2(M)\big)=0$ the Morse-Bott differential is of the form $\p=\p_0+$ auxiliary Morse trajectories which immediately implies the Theorem.
\end{Rmk}

\section{A conjectural explanation}\label{sec:conj_explanation}

Let $V$ be a Liouville domain, i.e.~a compact exact symplectic manifold with contact type boundary. We recall one of the main theorems by Cieliebak-Frauenfelder-Oancea in \cite{Cieliebak_Frauenfelder_Oancea_Rabinowitz_Floer_homology_and_symplectic_homology}.  There is a long exact sequence between symplectic (co-)homology $\SH$ and Rabinowitz Floer homology $\RFH$ as follows.
\beq\label{eq:CFO_les}
\cdots\pf \SH^{-*}(V)\pf\SH_*(V)\pf\RFH_*(\p V,V)\pf\SH_{-*-1}(V)\pf\cdots\;.
\eeq
Moreover, the map $\SH^{-*}(V)\to\SH_*(V)$ splits as
\beq
\xymatrix{\SH^{-*}(V)\ar[d]\ar[rr]&&\SH_*(V)\\
\H^{-*+d}(V,\p V)\ar[r]^-{\mathrm{PD}}&\H_{*+d}(V)\ar[r]^-{\mathrm{incl}_*}&\H_{*+d}(V,\p V)\ar[u]
}
\eeq
where $\mathrm{PD}$ denotes Poincare duality and $d=\tfrac{1}{2}\dim V$. As observed by Ritter in \cite{Ritter_Topological_quantum_field_theory_structure_on_symplectic_cohomology} this long exact sequence together with the fact that $\SH$ is a ring with unity leads to the statement
\beq\label{eq:equivalence_of_vanishing}
\SH^*(V)=0\quad\Longleftrightarrow\quad\SH_*(V)=0\quad\Longleftrightarrow\quad\RFH_*(\p V,V)=0\;.
\eeq
Note that our (co-)homology and grading conventions match with the ones in \cite{Cieliebak_Frauenfelder_Oancea_Rabinowitz_Floer_homology_and_symplectic_homology}.

In \cite{Oancea_Fibered_symplectic_cohomology_and_the_Leray_Serre_spectral_sequence} Oancea proves $\SH_*(E)=0$ for negative line bundles $\wp:E\to M$ under the condition that $(E,\Om)$ is symplectically aspherical. We point out that even in the symplectically aspherical case $E$ is not a Liouville manifold since $[\Om]\neq0\in\H^2(E)$. Ritter computes in \cite[Theorem 1]{Ritter_Floer_theory_for_negative_line_bundles_via_Gromov_Witten_invariants} for more general negative line bundles
\beq\label{eqn:Ritter}
\SH_*(E)\cong\QH_{*+d}(E,\Sigma)/\ker r^k
\eeq
where $d=\tfrac{1}{2}\dim E$, $\QH_*(E,\Sigma)$ is the relative quantum homology of the disk bundle inside $E$ with boundary $\Sigma$ and $r:\QH_*(E)\to\QH_{*-2}(E)$ is the map given by quantum intersection product with $\mathrm{PD}(\wp^*c_1^E)\in\QH_{2d-2}(E)$. Finally \eqref{eqn:Ritter} holds for any  $k\geq \dim\H_*(M)$. In particular, 
\beq
\SH_*(E)=0\quad\Longleftrightarrow\quad \wp^*c_1^E \text{ is nilpotent in }\QH^*(E)\;
\eeq
which generalizes Oancea's computation. Considering for instance the bundle $\wp:\O(-n)\to\CP^m$ it follows that 
\beq
\SH_*(\O(-n))\neq0\;,
\eeq 
for $n\leq m$, see \cite[Section 1.5]{Ritter_Floer_theory_for_negative_line_bundles_via_Gromov_Witten_invariants}. On the other hand, Theorem \ref{thm:main} applies since $(\CP^m,n\om_\mathrm{FS})$ is monotone with $c=\tfrac{m+1}{n}$ with corresponding negative line bundle $\O(-n)$, i.e.~we conclude
\beq
\RFH_*(\Sigma,\O(-n))=0\;.
\eeq
In this case $\Sigma$ is a Lens space. This is, of course, no contradiction to \eqref{eq:equivalence_of_vanishing} since the space $\O(-n)$ is \emph{not} a Liouville manifold. Also, $(\CP^m,n\om_\mathrm{FS})$ is not symplectically aspherical.

We offer the following conjectural explanation of Ritter's result \eqref{eqn:Ritter} in terms of the long exact sequence \eqref{eq:CFO_les} from \cite{Cieliebak_Frauenfelder_Oancea_Rabinowitz_Floer_homology_and_symplectic_homology} and Theorem \ref{thm:main}. We claim that the long exact sequence \eqref{eq:CFO_les} remains valid for negative line bundles $E$ (and probably even more generally) but the splitting of the map $\SH^{-*}(V)\to\SH_*(V)$ needs to be corrected as follows.
\beq
\xymatrix{\SH^{-*}(E)\ar[d]\ar[rr]&&\SH_*(E)\ar[r]&\RFH_*(\Sigma,E)\\
\QH^{-*+d}(E,\Sigma)\ar[r]^-{\mathrm{PD}}&\QH_{*+d}(E)\ar[r]^-{\mathrm{incl}_*}&\QH_{*+d}(E,\Sigma)\ar[u]_{c_*}
}
\eeq
Here, as in \cite{Ritter_Floer_theory_for_negative_line_bundles_via_Gromov_Witten_invariants}, we identify $\QH_*(E,\Sigma)$ as Floer homology of a Hamiltonian with very small slope at infinity or equivalently as symplectic homology in the action window $(-\varepsilon,\varepsilon)$. Then $c_*$ is just a continuation homomorphism induced by a canonical inclusion map. We refer to \cite{Ritter_Floer_theory_for_negative_line_bundles_via_Gromov_Witten_invariants} for details. In particular, if $\RFH_*(\Sigma,E)=0$ then the map $c_*$ is surjective and
\beq
\SH_*(E)\cong\QH_{*+d}(E,\Sigma)/\ker c_*\;.
\eeq
Ritter's important observation in  \cite{Ritter_Floer_theory_for_negative_line_bundles_via_Gromov_Witten_invariants}  is that $c_*$ is indeed surjective and can be identified with $r^k$ for large $k$ under his assumptions.

As mentioned above Ritter's and the present result holds for the bundle $\O(-n)\to\CP^m$. In fact, from inspection of Ritter's article \cite{Ritter_Floer_theory_for_negative_line_bundles_via_Gromov_Witten_invariants} it seems that Theorem \ref{thm:main} applies to all examples Ritter considers.

%

\end{document}